\documentclass[12pt, reqno]{amsart}

\usepackage{amsmath, amsthm, amscd, amsfonts, amssymb, graphicx, color}
\usepackage[bookmarksnumbered, colorlinks, plainpages]{hyperref}

\makeatletter \oddsidemargin.9375in \evensidemargin \oddsidemargin
 \newtheorem{theorem}{Theorem}[section]

\newtheorem{lem}[theorem]{Lemma}

\newtheorem{defn}[theorem]{Definition}

\newtheorem{exam}[theorem]{Example}

\numberwithin{equation}{section}

\newcommand{\BC}{{\Bbb C}}
\newcommand{\BN}{{\Bbb N}}
\newcommand{\BR}{{\Bbb R}}

\newcommand{\m}{{\rm max}}
\newcommand{\M}{{\mathcal{M}}}

\newcommand{\A}{{\mathcal{A}}}

\newcommand{\B}{{\mathcal{B}}}

\newcommand{\Ch}{{\rm Ch}}

\begin{document}
\title[Composition in modulus maps on semigroups of continuous functions]{Composition in modulus maps on semigroups of continuous functions}

\author{Bagher Jafarzadeh and  Fereshteh Sady$^1$}



\maketitle
\begin{center}

\address{{\em   Department of Pure
Mathematics, Faculty of \\ Mathematical Sciences,
 Tarbiat Modares University,\\ Tehran, 14115-134, Iran}}

\vspace*{.25cm}
  \email{b.jafarzadeh@modares.ac.ir, sady@modares.ac.ir}
\end{center}
\footnote{$^1$ Corresponding author}

\maketitle

\begin{abstract}
For locally compact Hausdorff spaces $X$ and $Y$, and function
algebras $A$ and $B$ on $X$ and $Y$, respectively,  surjections
$T:A \longrightarrow B$ satisfying norm multiplicative condition
$\|Tf\, Tg\|_Y =\|fg\|_X$, $f,g\in A$, with respect to the
supremum norms, and those satisfying
$\||Tf|+|Tg|\|_Y=\||f|+|g|\|_X$ have been extensively studied.
Motivated by this, we consider certain (multiplicative or
additive) subsemigroups $A$ and $B$  of $C_0(X)$ and $C_0(Y)$,
respectively, and study  surjections $T: A \longrightarrow B$
satisfying the norm condition $\rho(Tf, Tg)=\rho(f,g)$, $f,g \in
A$, for some class of two variable positive functions $\rho$. It
is shown that $T$ is also a composition in modulus map.
 \end{abstract}
\section{Introduction}
The interaction between different structures of a space has been
studied in many settings. In the context of function algebras, the
classical Banach-Stone theorem and its generalizations
characterize isometries between certain algebras of continuous
functions as multiples by a continuous function of an algebra
isomorphisms. By the Mazur-Ulam theorem, any surjective isometry
between real normed spaces, preserves midpoints, and so it is a
real-linear map up to a translation. That is, surjective
isometries reveal real vector space structures of the normed
spaces.

Multiplicative version of the Banach-Stone theorem characterizes
surjections $T:A \longrightarrow B$, not assumed to be linear,
between different subsets $A$ and $B$ of $C_0(X)$ and $C_0(Y)$,
for locally compact Hausdorff spaces $X$ and $Y$, which are
multiplicatively norm-preserving, i.e.  $\|Tf\;Tg\|_Y=\|fg\|_X$
holds  for all $f,g\in A$. The notations $\|\cdot\|_X$ and
$\|\cdot\|_Y$ stand for the supremum norms. In the setting of
function algebras, such a map $T$ is a  composition in modulus
map, i.e. there exists a homeomorphism $\Phi: \Ch(B)
\longrightarrow \Ch(A)$ between the Choquet boundaries of $A$ and
$B$ such that $|Tf(y)|=|f(\Phi(y))|$ for all $f\in A$ and $y\in
\Ch(B)$, see \cite{Lam-Lt-Tonev}. The idea of considering such
maps comes from Molnar's result \cite{Molnar} concerning
multiplicatively spectrum preserving maps between operator
algebras and also $C(X)$-spaces. The result has been improved in
various directions for many different settings
 such as (Banach) function  algebras and their multiplicative subsets, see for example
 \cite{Hatori, Hat-mon, Hat-etal, Hos-Sady, Rao-Roy1, Rao-Roy2} and also the survey \cite{Survey}.
Norm additive in modulus maps between function algebras  has been
studied in \cite{Tonev-Yates}. Such mappings satisfy the norm
condition $\|\,|Tf|+|Tg|\,\|_Y=\|\,|f|+|g|\,\|_X$, and it is shown
in \cite{Tonev-Yates} that they are also composition in modulus
maps. We note that for positive cones of spaces of functions, the
above norm condition is, in fact, the norm additive condition
$\|Tf+Tg\|_Y=\|f+g\|_X$. Motivated by the Mazur-Ulam theorem, the
authors of \cite{Molnar-Z} consider a more general problem for
positive cones of operator algebras and positive cones of
subalgebras of continuous functions. Indeed, by introducing the
notation of mean, they study surjections $T$ between operator
algebras and between positive cones of subalgebras of continuous
functions satisfying the norm condition
$\|\M(Tf,Tg)\|=\|\M(f,g)\|$ with respect to a mean $\M$. 
A similar problem has been considered in the recent work \cite{Jaf-Sady} of the
authors.

In
\cite{Hat-etal2}, Hatori et. al. introduced the notations of
subdistances, metricoid spaces and  midpoint of the elements of
metricoid spaces, and then give some Mazur-Ulam type theorem.  In
particular, for a compact Hausdorff space $X$, they characterize
surjective maps $T$ on the set of strictly positive functions in
$C(X)$  preserving one of the subdistances
\[ \delta_+(f,g)=\left\|\frac{f}{g}-1\right\|_{X}+\left\|\frac{g}{f}-1\right\|_{X} \]
and
\[\delta_{{\rm max}}(f,g)={\rm max}\left(\left\|\frac{f}{g}-1\right\|_{X},
\left\|\frac{g}{f}-1\right\|_{X}\right).\] Motivated by the above
results, in this paper we consider two variable positive functions
$\rho_+$ and $\rho_{\rm max}$ defined by
\[ \rho_+(f,g)=\|\varphi(f,g)\|_X+\|\varphi(g,f)\|_X\]
and
\[\rho_{\rm max}(f,g)={\rm max}( \|\varphi(f,g)\|_X,\|\varphi(g,f)\|_X)\]
for $f,g\in C_0(X)$, where $ X $ is a locally compact Hausdorff
space and $\varphi: \BC \times \BC \longrightarrow \BR^+$ is a
certain continuous two variable function. Here for $f,g\in
C_0(X)$,  $\varphi(f,g)(x)=\varphi(f(x),g(x))$, $x\in X$. We study
surjections $T:A \longrightarrow B$ between some (multiplicative
or additive) semigroups $A$ and $B$ of continuous functions on
locally compact Hausdorff spaces $X$ and $Y$, respectively, such
that $\rho(Tf,Tg)=\rho(f,g)$, $f,g\in A$, where $\rho\in
\{\rho_{\rm max}, \rho_+\}$. It is shown that such a map $T$ is
also a composition in modulus map (Theorems \ref{main1} and
\ref{main2}).

\section{Preliminaries}
For a locally compact Hausdorff space $X$, $C_b(X)$ is the Banach
space of bounded continuous complex-valued functions on $X$ with
the supremum norm $\|\cdot \|_X$ and $C_0(X)$ is the closed
subalgebra  of $C_b(X)$ consisting of continuous functions
vanishing at infinity. A {\em function algebra} on $X$ is a closed
subalgebra $A$ of $C_0(X)$ which strongly separates the points of
$X$, that is, for any distinct points $x,y\in X$, there exists $f\in
A$ with $f(x)\neq f(y)$, and for each point $x\in X$, there exists
$g\in A$ with $g(x)\neq 0$.

For a subset $A$ of $C_0(X)$, a point $x\in X$ is called a {\em
strong boundary point} of $A$ if for each $\epsilon >0$ and
neighborhood $V$ of $x$, there exists $f\in A $ such that $
f(x)=1=\|f\|_X$ and $|f|<\epsilon $ on $X\setminus V $. We denote
the set of all strong boundary points of $A$ by $\delta(A)$. For a
point $x\in X$, the evaluation functional $e_x:A \longrightarrow
\BC$ is defined by $e_x(f)=f(x)$, $f\in A$.  For a subspace  $A$
of $C_0(X)$, the \emph{Choquet boundary} of $A$, denoted by $\Ch(A)$,
consists of all points $x\in X$ such that $e_x$ is an extreme
point of the unit ball of $A^*$. It is well known that $\Ch(A)$ is
a boundary for $A$, that is, for each $f\in A$, there exists a point
$x\in \Ch(A)$ such that $|f(x)|=\|f\|_X$, see \cite[Page
184]{Taylor}. In general, $\delta(A) \subseteq \Ch(A)$ (see
\cite[Lemma 3.1]{Jam-Sady}) and if  $A$ is a function algebra,
then  $\delta(A)=\Ch(A)$ (see \cite[Theorem 4.7.22]{Leib} for
compact case and \cite[Theorem 2.1]{Rao-Roy2} for general case).

Let $X$ be a locally compact Hausdorff space and $A$ be a subset
of $C_0(X)$. For a point $x_0\in X$, we fix the following
notations
\[ V_{x_0}(A)=\{f\in A:f(x_0)=1=\|f\|_X\}, \;\; F_{x_0}(A)=\{f\in A: |f(x_0)|=1=\|f\|_X\}.\]
Clearly, for $x_0\in \delta(A)$, these  sets  are nonempty.
Meanwhile, for $x_1,x_2\in \delta(A)$, each of the inclusions
$V_{x_1}(A) \subseteq V_{x_2}(A)$ and $F_{x_1}(A) \subseteq
F_{x_2}(A)$ implies that  $x_1=x_2$.

For $f\in C_0(X)$, we also set $M(f)=\{x\in X: |f(x)|=\|f\|_X\}$.
The notation $A_+$ is used for the  set of positive elements  of
$A$, i.e. $A_+=\{f\in A : f(x)\geq 0 \text{ for all } x\in X\}$.
We also put $|A|=\{|f|: f \in A\}$.

For a locally compact Hausdorff space $X$ and a subspace $A$ of
$C_0(X)$, a  function $f\in A$ with $\|f\|_X=1$ is called a {\em
peaking function} of $A$ if for each $x\in X$, either $|f(x)|<1$
or $f(x)=1$. A closed subset $F$ of $X$ is a {\em peak set} of $A$
if there exists a peaking function $f\in A$ such that  $F=\{x\in X: f(x)=1\}$. It
is well known that in a function algebra $A$ on $X$, each nonempty
intersection of peak sets of $A$ intersects $\Ch(A)$.

\section {Certain two variable functions}
In this section, we consider a positive two variable function
$\varphi$ with a property called (inc) and provide some required
lemmas which will be used in the next sections.

 Let $\varphi: \BC \times \BC \longrightarrow \BR^+$ be a continuous map. We define the
following increasing property:
\renewcommand{\labelitemi}{(inc)}
\begin{itemize} \item
$\varphi$ is strictly increasing in modulus  with respect to both
variables, in the sense that for $s_1,s_2\in\BC$, if $|s_1|\le
|s_2|$, then  $\varphi(s_1,t)\le \varphi(s_2,t)$ and
$\varphi(t,s_1)\le \varphi(t,s_2)$ for all $t\in \BC $, and the
same implication holds for all $t\in \BC\backslash \{0\}$ if we
replace ''$\le $'' by ''$<$''.
\end{itemize}
Examples of two variable functions satisfying (inc) are as
follows.

\begin{exam}{\rm
(i) For strictly positive scalars $a$ and $b$, the maps $\varphi_{a,b}$ and $
\psi_{a,b}$ on $\BC \times \BC$  defined by $\varphi
_{a,b}(s,t)=a\,|s|+b\,|t|$ and   $\psi_{a,b}(s,t)=|s|^a \, |t|^b$
satisfy (inc).

(ii) For strictly positive scalar $p$, the map $\varphi_p: \BC
\times \BC \longrightarrow \BR^+$ defined by $\varphi
_p(s,t)=(|s|^p+|t|^p)^{1/p}$ satisfies (inc).

(iii) If $\varphi, \psi: \BC \times \BC \longrightarrow \BR^+$ are
continuous maps such that $\varphi$ is (not necessarily strictly)
increasing in modulus and $\psi$ satisfies (inc), then
$\varphi+\psi$ also satisfies (inc). In particular, the following
maps satisfy (inc)
\begin{align*}
 \varphi(s,t) &=a|s|+b|t|+|s|^c |t|^d \;\; {\rm  for} \; a,b,c,d>0,\\
 \varphi(s,t)&= {\rm max}(|s|,|t|)+ |s|+|t|, \; \;  \psi(s,t)={\rm
max}(|s|,|t|)+|st|,\\
 \varphi(s,t)&= {\rm min}(|s|,|t|)+ |s|+|t|, \; \; \psi(s,t)={\rm
min}(|s|,|t|)+|st|.
\end{align*}
 }
\end{exam}

For functions $f,g \in C_b(X)$, the continuous function
$\varphi(f,g)$ on $X$ is defined by $\varphi(f,g)(x)=
\varphi(f(x),g(x))$ for all $x\in X$.

Next lemma is easily verified. For the sake of completeness, we state and prove it here.

\begin{lem} \label{lem1}
Let $X$ be a locally compact Hausdorff space and let $\varphi: \BC
\times \BC \longrightarrow \BR^+$ be  a continuous map satisfying
{\rm (inc)}.

{\rm (i)} If $\varphi(0,0)=0$, then for all $f,g\in C_0(X)$, we
have $\varphi(f,g)\in C_0(X)$.

{\rm (ii)} For $s,t\in \BC$, we have $\varphi(s,t)=\varphi(|s|,|t|)$.

{\rm (iii)} For $a,b\in \BC$ and $c,d\in \BC \backslash
\{0\}$, if $|a|<|c|$ and $|b|\le |d|$, then $\varphi(a,b)<
\varphi(c,d)$ and $\varphi(b,a)<\varphi(d,c)$.

{\rm (iv)} For  $f,g\in C_0(X)$, if  $r,s>0$ such that
$\varphi(f(x),g(x))< \varphi(r,s)$ for all $x\in X$, then
$\|\varphi(f,g)\|_X<\varphi(r,s)$.
\end{lem}
\begin{proof}
(i)-(iii) are easily verified by using (inc).

(iv) Let $X_\infty$ be the one point compactification of $X$. Then
 $\varphi(f,g)$ is an element of $C(X_\infty)$ and since $\varphi(0,0)\le \varphi(0,s)<\varphi(r,s)$, it follows
from the hypothesis that $\varphi(f(x),g(x))< \varphi(r,s)$ for all
$x\in X_\infty$. Hence
\[ \|\varphi(f,g)\|_X\le
\|\varphi(f,g)\|_{X_\infty}<\varphi(r,s).\qedhere\]
\end{proof}

\begin{defn}\label{dfn}
Let $\varphi: \BC \times \BC \longrightarrow \BR^+$ be a
continuous map satisfying {\rm (inc)}. For a locally compact
Hausdorff space $X$, we define $\rho_+, \rho_{{\rm max}}: C_b(X)
\times C_b(X) \longrightarrow \BR^+$ by
\[ \rho_+(f,g)=\|\varphi(f,g)\|_X+\|\varphi(g,f)\|_X\]
and
\[ \rho_{{\rm max}} (f,g)=\max (\|\varphi(f,g)\|_X,\|\varphi(g,f)\|_X).\]
\end{defn}
In the rest of this section, we assume that the continuous map
$\varphi:\BC \times \BC \longrightarrow \BR^+$ satisfies (inc), and
positive functions $\rho_+$ and $\rho_{{\rm max}}$ are as  above.

Next lemma states some simple observations about $\rho_+$ and
$\rho_{\rm max}$.

\begin{lem}
Let $X$ be a locally compact Hausdorff space and $\rho\in \{\rho_+, \rho_{\rm max}\}$. Then the following
statements hold.

{\rm (i)} For $f_1,g_1,f_2,g_2\in C_b(X)$, the
inequality $\rho(f_1,g_1)<\rho(f_2,g_2)$ implies that either
$\|\varphi(f_1,g_1)\|_X< \|\varphi(f_2,g_2)\|_X$ or
$\|\varphi(g_1,f_1)\|_X<\|\varphi(g_2,f_2)\|_X$.

{\rm (ii)} For $f,g \in C_b(X)$ and $r,s>0$, if
$\rho(f,g)<\rho(r,s)$, then for each $x\in X$, we have either
$|f(x)|<r$ or $|g(x)|<s$.
\end{lem}

\begin{lem}\label{romax}
Let $X$ be a locally compact Hausdorff space and  $A$ be a subset
of $C_0(X)$. Let $x_0\in \delta(A)$ and $f\in A$. Then for each
$\epsilon>0$,  there exists $h\in V_{x_0}(A)$ such that
\[ \rho_{\rm max}(f,\|f\|_Xh) \le \rho_{\rm max}(|f(x_0)|+\epsilon, \|f\|_X)\]
\end{lem}
\begin{proof}
 The inequality is obvious if $f=0$. Hence we assume that $f\neq 0$.

  For  $0<\epsilon \le 1$, consider the open neighborhood $V=\{x\in X: |f(x)-f(x_0)|<\epsilon\}$ of $x_0$.
 Since $x_0\in \delta(A)$,  there exists $h\in V_{x_0}(A)$ such that $|h|< \frac{\epsilon}{\|f\|_X}$ on $X\backslash V$.
 Using condition (inc) on $\varphi$,  we get  $\varphi(f(x),\|f\|_X h(x))\le \varphi(|f(x_0)|+\epsilon,\|f\|_X)$
 for all $x\in V$, and moreover, $\varphi(f(x), \|f\|_X h(x))\le
\varphi(\|f\|_X, \epsilon)\le \varphi(\|f\|_X, |f(x_0)|+\epsilon)$
for all $x\in X\backslash V$. Therefore,  $\|\varphi(f,\|f\|_X
h)\|_X \le \rho_{\m}(|f(x_0)|+\epsilon, \|f\|_X)$. A similar
discussion shows that $\|\varphi(\|f\|_X h,f)\|_X \le
\rho_{\m}(|f(x_0)|+\epsilon, \|f\|_X)$. Hence
\[\rho_{{\rm max}}(f,\|f\|_X h)\le \rho_{{\rm max}}(|f(x_0)|+\epsilon, \|f\|_X).\qedhere\]
\end{proof}

\begin{lem}\label{ro+}
Let $X$ be a locally compact Hausdorff space and $A$ be a subset
of $C_0(X)$. Let  $x_0\in \delta(A)$,  $f\in A$ and $\epsilon>0$.

{\rm (i)} If $\varphi$ satisfies  the additional condition
$\varphi(t,0)=0=\varphi(0,t)$ for all $t> 0$, then there exists
$h\in V_{x_0}(A)$ such that
\[ \rho_+(f,h) < \rho_+(|f(x_0)|+\epsilon,1).\]

{\rm (ii)} If for all $a>0$, $\varphi(t,a)\to \infty$ and
$\varphi(a,t)\to \infty$ as $t\to \infty$, then  there exist
$\lambda>0$ and $h\in V_{x_0}(A)$ such that
\[ \rho_+(f,\lambda h) < \rho_+(|f(x_0)|+\epsilon,\lambda).\]
\end{lem}
\begin{proof}
(i) Since $\varphi(t,0)=\varphi(0,t)=0$ for all $t\ge 0$, it
follows from the continuity of $\varphi$ that there is a small enough
$0<\epsilon'<\epsilon $ such that $\epsilon' \|f\|_X<1$,
$\varphi(\|f\|_X,\epsilon'\|f\|_X)<\varphi(|f(x_0)|+\epsilon,1)
$ and
$\varphi(\epsilon'\|f\|_X,\|f\|_X)<\varphi(1,|f(x_0)|+\epsilon)
$. We should  note that
$\varphi(1,|f(x_0)|+\epsilon)>\varphi(0,0)\ge 0$. Let $U$ be an
open neighborhood of $x_0$ such that $|f(x)|< |f(x_0)|+\epsilon$
for all $x\in U$. Since $x_0\in\delta(A) $, we can find $ h\in A $
such that $h(x_0)=1=\|h\|_X $ and $ |h|<\epsilon'\|f\|_X $ on $
X\setminus U $. Using (inc), for each $x\in X\backslash U$, we have
\[\varphi(f(x),h(x))\le  \varphi(\|f\|_X,\epsilon'\|f\|_X)<\varphi(|f(x_0)|+\epsilon,1). \]
On the other hand, for each $x\in U$, using Lemma \ref{lem1}(iii), we
have
\[\varphi(f(x),h(x))<\varphi(|f(x_0)|+\epsilon ,1).\]
Therefore, by Lemma \ref{lem1}(iv),
$\|\varphi(f,h)\|_X<\varphi(|f(x_0)|+\epsilon,1)$. Similarly, it follows that
$\|\varphi(h,f)\|_X<\varphi(1,|f(x_0)|+\epsilon)$.
Hence
\[\rho_+(f,h) < \rho_+(|f(x_0)|+\epsilon,1),\]
as desired.

(ii)  Clearly, in this case, we can choose $\lambda>\|f\|_X $ such
that $\varphi(\|f\|_X,\|f\|_X) < \varphi(\lambda,
|f(x_0)|+\epsilon)$ and
$\varphi(\|f\|_X,\|f\|_X)<\varphi(|f(x_0)|+\epsilon, \lambda)$.
Let $U$ be an open neighborhood of $x_0$ such that
$|f(x)|<|f(x_0)|+\epsilon$ for all $x\in U$. Since
$x_0\in\delta(A) $, there exists $ h\in A $ such that $
h(x_0)=1=\|h\|_X $ and $ |h|<\frac{1}{\lambda}\|f\|_X $ on $
X\setminus U $. Hence for each $x\in U$, $ \varphi(f(x),\lambda
h(x))< \varphi(|f(x_0)|+\epsilon ,\lambda) $,     and for $
x\in X \backslash U $,
\[\varphi(f(x),\lambda h(x))\le \varphi(\|f\|_X,\|f\|_X)<
\varphi(|f(x_0)|+\epsilon ,\lambda).\] Therefore, by Lemma
\ref{lem1}(iv), we have $\|\varphi(f,\lambda h)\|_X<
\varphi(|f(x_0)|+\epsilon ,\lambda)$. Similarly, one can show
that $\|\varphi(\lambda
h,f)\|_X<\varphi(\lambda,|f(x_0)|+\epsilon)$. Hence
\[ \rho_+(f,\lambda h)<\rho_+(|f(x_0)|+\epsilon, \lambda).\qedhere\]
\end{proof}

\section{Additive  semigroups of continuous functions}
In this section, we assume that $\varphi: \BC \times \BC
\longrightarrow \BR^+$ is  a continuous map satisfying (inc) and also the following
condition:
\renewcommand{\labelitemi}{(con)}
\begin{itemize} \item
 For every $n\in\BN$ and $s_1,s_2,...,s_n,t\in
\BC$, $\varphi(\frac{1}{n}\Sigma_{i=1}^{n}s_i,t)\le
\frac{1}{n}\Sigma_{i=1}^{n}\varphi(s_i,t)$ and $\varphi(t,\frac{1}{n}\Sigma_{i=1}^{n}s_i)\le
\frac{1}{n}\Sigma_{i=1}^{n}\varphi(t,s_i)$.
\end{itemize}
\vspace*{.2cm} We also consider $\rho_+$ and $\rho_{{\rm max}}$ as in
Definition \ref{dfn} and study surjections whose domains are
certain additive semigroups of continuous functions and preserve
$\rho_+$ and $\rho_{{\rm max}}$.

 Before stating our result, we give some examples of
such two variable functions $\varphi$.
 \begin{exam}{\rm
{\rm (i)} For each $a,b>0$, the map $\varphi_{a,b}$, defined by
$\varphi_{a,b}(s,t)=a|s|+b|t|$ for $s,t\in \BC$, satisfies both
(inc) and (con).

{\rm (ii)} The map $\psi(s,t)=|st|$, $s,t\in
\BC$, satisfies (inc) and (con).

 {\rm (iii)} The sum $\varphi_1+\varphi_2$  of continuous
maps $\varphi_i: \BC \times \BC \longrightarrow \BR^+$, $i=1,2$,
satisfying (inc) and (con), again satisfies these conditions.

{\rm (iv)} The continuous maps $\varphi(s,t)= {\rm max}(|s|,|t|)+
|s|+|t|$ and $\psi(s,t)={\rm max}(|s|,|t|)+|st|$ satisfy (inc) and
(con). In general, if $\varphi, \psi: \BC \times \BC
\longrightarrow \BR^+$ are continuous maps such that $\varphi$ is
(not necessarily strictly) increasing in modulus, satisfying
(con), and $\psi$ satisfies both (inc) and (con), then
$\varphi+\psi$ satisfies both (inc) and (con). }
\end{exam}
Next theorem is our main result in this section.
\begin{theorem}\label{main1}
Let $\varphi: \BC \times \BC \longrightarrow \BR^+$ be a
continuous map satisfying {\rm (inc)} and {\rm (con)}. Let $X$ and
$Y$ be locally compact Hausdorff spaces, and $A$ and $B$ be either
subspaces of $C_0(X)$ and $C_0(Y)$, or  positive cones of some
subspaces of $C_0(X)$ and $C_0(Y)$, respectively. Suppose that
$\delta(A)=\Ch(\A)$ and $\delta(B)=\Ch(\B)$ for some function
algebras $ \A $ and $ \B$ on $X$ and $Y$, respectively, with $|A|
\subseteq |\A|$ and $|B|\subseteq |\B|$. Let $\rho\in \{\rho_+,
\rho_{\rm max}\}$ and $T: A\longrightarrow B$ be a surjective map
satisfying
\[ \rho(Tf,Tg)= \rho(f,g) \qquad (f,g\in A).  \]
Then $T$ induces a bijection $\Phi: \Ch(\B)\longrightarrow
\Ch(\A)$ between the Choquet boundaries of $\A$ and $\B$.
Moreover,

{\rm (i)} If $\rho=\rho_{\rm max}$, then $ \Phi $ is a
homeomorphism and $|Tf(y)|=|f(\Phi(y))|$ for all $f\in A$ and $y\in \Ch(\B)$.

{\rm (ii)} If $\rho= \rho_+$ and either
\begin{itemize}
\item[(a)] $\varphi(t,0)=0=\varphi(0,t)$ for all $t\ge 0$, or

\item[(b)] $\varphi(t,a)\to \infty$ and $\varphi(a,t)\to \infty$
as $t\to \infty$ for all $a>0$,
\end{itemize}
holds, then $ \Phi $ is a homeomorphism and
$|Tf(y)|=|f(\Phi(y))|$ for all $f\in A$ and $y\in\Ch(\B)$.
\end{theorem}
We prove the theorem through the subsequent lemmas.

In what follows, we assume that $X,Y$ and $A,B$ are as in Theorem
\ref{main1}  and $T:A\longrightarrow B$ is a surjection satisfying
\[ \rho(Tf,Tg)= \rho(f,g) \qquad (f,g\in A),  \]
where $\rho \in \{\rho_+,\rho_{\rm max}\}$.

\begin{lem}\label{cap}
{\rm (i)} For each $f\in A$, we have $\|Tf\|_Y=\|f\|_X$.

{\rm (ii)} Let $r>0$. Then for any convex subset $C$ of the sphere
$S_r(A)=\{f\in A: \|f\|_X=r\}$, we have  $\cap_{f\in C} M(Tf)\neq
\emptyset$. Similarly,
 for any convex subset $C'$ of the sphere $S_r(B)=\{g\in B: \|g\|_Y=r\}$, we have  $\cap_{Tf\in C'} M(f)\neq \emptyset$.
\end{lem}
\begin{proof}
 (i) For each $f\in A$, it follows from (inc) that  $\|\varphi(f,f)\|_X=\varphi(\|f\|_X,\|f\|_X)$ and $\|\varphi(Tf,Tf)\|_Y=\varphi(\|Tf\|_Y,\|Tf\|_Y)$. Hence
 for arbitrary $x_0\in M(f)$ and $y_0\in M(Tf)$, we have
 \[ \rho(|Tf(y_0)|,|Tf(y_0)|)=\rho(Tf,Tf)=\rho(f,f)=\rho(|f(x_0)|, |f(x_0)|).\]
 Thus $|Tf(y_0)|=|f(x_0)|$, and consequently $\|Tf\|_Y =\|f\|_X$.

 (ii) Let $C$ be a convex subset of $S_r(A)$. It suffices to show that the family $\{M(Tf): f\in C\}$ of compact subsets of $Y$
 has finite intersection property. Let $f_1,..., f_n\in C$. Since $C$ is a convex subset of $S_r$,
 we have $h= \frac{1}{n} \Sigma_{i=1}^{n} f_i\in C$.
 By (i), $\|Th\|_Y=\|h\|_Y=r$.  Since $ \Ch(\B) $ is a boundary for $ \B $ and $|B|\subseteq |\B|$, we can choose $y_0\in \Ch(\B)$ such that $|Th(y_0)|=r=\|h\|_X$.  We claim that $y_0\in M(Tf_i)$ for
$i=1,...,n$. Assume on the contrary that $|Tf_j(y_0)|<r$ for some
$1\le j\le n$. Then there exists a neighborhood $V$ of $y_0$ such
that $|Tf_j|<r$ on $V$. Since $y_0\in \delta(B)$, we can find
$h'\in A$ such that $Th'(y_0)=1=\|Th'\|_Y$ and $|Th'|< 1$ on $Y \backslash
V$. Since $ |Tf_j(y)|<r $ and $ |Th'(y)|\le1 $ for all $ y\in V $, and $ |Tf_j(y)|\le r $ and $ |Th'(y)|<1 $ for all $ y\in Y \backslash V $, it follows by Lemma \ref{lem1}(iii) that $\varphi(Tf_j(y), Th'(y))<
\varphi(r,1)$  and $\varphi(Th'(y),Tf_j(y))< \varphi(1,r)$ hold
for all $y\in Y$. Hence according to Lemma \ref{lem1}(iv), we have $\|\varphi(Tf_j, Th')\|_Y< \varphi(r,1)$
and $\|\varphi(Th',Tf_j)\|_Y< \varphi(1,r)$, that is
$\rho(Tf_j,Th')< \rho(r,1)$. The  hypotheses imply that
\[  \rho(f_j,h') < \rho( r,1).\]
Thus, for each $ x\in X $, at least one of the inequalities
$|f_j(x)|<r $ and $|h'(x)|<1$ holds. Now it follows from (con) that
for each $x\in X$,
\begin{eqnarray*}
\varphi(h(x),h'(x)) &\le& \frac{1}{n} \varphi(f_j(x),
h'(x))+\frac{1}{n} \Sigma_{i\neq j} \varphi(f_i(x), h'(x))\\
&<& \frac{1}{n} \varphi(r,1)+
\frac{n-1}{n}\varphi(r,1)=\varphi(r,1)
\end{eqnarray*}
and similarly $ \varphi(h'(x),h(x))<\varphi(1,r) $. Hence
\begin{align*}
    \rho(Th,Th')=\rho(h,h')<\rho(r,1),
\end{align*}
and consequently
\[ \rho(r,1)=\rho(Th(y_0),Th'(y_0))\le \rho(Th,Th')<\rho(r,1) \]
which is a contradiction. This argument shows that $y_0\in \cap_{i=1}^n M(Tf_i)$, as desired.

The other part is similarly proven.
\end{proof}

In this section, for $y_0\in \Ch(\B)$ and $r>0$, we set $I^r_{y_0}= \cap_{Tf\in
rV_{y_0}(B)}M(f)$. Similarly, for $x_0\in \Ch(\A)$ and $r>0$, we set $J^r_{x_0}=\cap_{f\in rV_{x_0}(A)} M(Tf)$.

\begin{lem}\label{int}
Let $y_0\in \Ch(\B)$ and $x_0\in \Ch(\A)$. Then for each $r>0$, $I^r_{y_0} \cap \Ch(\A)\neq
\emptyset$ and $J^r_{x_0}\cap \Ch(\B) \neq \emptyset$.
\end{lem}
\begin{proof}
We prove the first assertion, the second one is proven in a
similar manner. Let $y_0\in \Ch(B)$. Since for each $r>0$, the set
$rV_{y_0}(B)$ is a convex subset of $S_r(B)=\{g\in B:
\|g\|_Y=r\}$, it follows from Lemma \ref{cap}(ii) that
$I^r_{y_0}\neq \emptyset$. Now, let $ x_0^r\in I^r_{y_0} $. Then
for each $f\in A$ with $Tf\in rV_{y_0}(B)$, we have $x_0^r\in
M(f)$. By assumption, for each $f\in A$ there exists $g\in \A$
with $|f|=|g|$, which implies that $M(f)=M(g)$. The maximum
modulus set $M(g)$ of $g$ contains a peak set of the function
algebra $\A$ containing $x_0^r$, hence using the fact that any
nonempty intersection of peak sets of $\A$ intersects $\Ch(\A)$,
we get $I^r_{y_0} \cap \Ch(\A)\neq\emptyset$.
\end{proof}

\begin{lem}\label{iff}
Let $x_0\in \Ch(\A)$ and $y_0\in \Ch(\B)$. Let $s>0$ be given. Then

{\rm (i)} $x_0\in I^1_{y_0}$ if and only if $y_0\in J^s_{x_0}$,

{\rm (ii)} $y_0\in J^1_{x_0}$ if and only if $x_0\in I^s_{y_0}$.
\end{lem}
\begin{proof}
(i) Assume first that  $x_0\in I^1_{y_0}$ and $y_0\notin
J^s_{x_0}$. Then, by the definition of $J^s_{x_0}$,
there exists $f\in sV_{x_0}(A)$ such that $|Tf(y_0)|<s$. We note
that $\|Tf\|_Y=\|f\|_X=s$. Since $ Tf $ is continuous and $|Tf(y_0)|<s$,
we can choose a neighborhood $U$
of $y_0$ in $Y$ such that $|Tf|< s$ on $U$. Since $y_0\in
\Ch(\B)=\delta(B)$ and $T$ is surjective, we can find a function
$h\in A$ with $Th\in V_{y_0}(B)$ such that $|Th|<1$ on $Y\backslash U$.
Thus for each $y\in Y$, at least one of the inequalities
$|Th(y)|<1$ and $|Tf(y)|<s$ holds. This easily implies, by Lemma \ref{lem1}(iii), that for each
$y\in Y$, we have $\varphi(Tf(y), Th(y))<\varphi(s,1)$, and using
Lemma \ref{lem1}(iv), we get $\|\varphi(Tf,Th)\|_Y< \varphi(s,1)$. Similar
argument shows that $\|\varphi(Th,Tf)\|_Y< \varphi(1,s)$, and
consequently
\[\rho(f,h)=\rho(Tf,Th)< \rho(s,1).\]
On the other hand, since $Th\in V_{y_0}(B)$ and $x_0\in I^1_{y_0}$, it follows that $|h(x_0)|=1=\|h\|_X$. Therefore,
\[\rho(s,1)=\rho(|f(x_0)|,|h(x_0)|)=\rho(f(x_0),h(x_0))\le \rho(f,h)< \rho(s,1),\]
a contradiction. The other implication is similarly proven.

(ii) It is proven by a similar argument in (i).
\end{proof}

\begin{lem}\label{singleton}
For each $y_0\in \Ch(\B)$,  there exists a point $x_0\in
\Ch(\A)$ such that for all $r>0$, $I_{y_0}^r \cap \Ch(\A)=\{x_0\}$
and $J_{x_0}^r \cap \Ch(\B)=\{y_0\}$.
\end{lem}
\begin{proof}
Let $y_0\in \Ch(\B)$ and let $x_0$ be an arbitrary element in
$I^1_{y_0}\cap \Ch(\A)$. Then,  using Lemma \ref{iff}(i), we get $y_0\in
J^1_{x_0}$. Part (ii) of this lemma shows that  $x_0\in I^r_{y_0}$
for all $r>0$. That is  $I_{y_0}^1 \cap \Ch(\A)\subseteq I_{y_0}^r
\cap \Ch(\A)$ for all $r>0$. Conversely, if $r>0$ and $x_0$ is an
arbitrary point of $I_{y_0}^r \cap \Ch(\A)$, then by Lemma \ref{iff}(ii),
we have $y_0\in J^1_{x_0}$, and so, using part (i) of this lemma,
we get $x_0\in I^1_{y_0}\cap \Ch(\A)$.
Therefore, $I_{y_0}^1 \cap \Ch(\A)\supseteq  I_{y_0}^r \cap
\Ch(\A)$, that is
\[I_{y_0}^1 \cap \Ch(\A)= I_{y_0}^r \cap \Ch(\A)\;\;{\rm  for\,\, all \,r>0}.\]
Hence it suffices to show that $I_{y_0}^1 \cap \Ch(\A)$ is a
singleton. Assume on the contrary that $x_0,x_1$ are distinct
points in this intersection. Choose disjoint neighborhoods $U$ and
$V$ of $x_0$ and $x_1$, respectively. As $x_0,x_1\in \delta(A)$,
we can find functions $f\in V_{x_0}(A)$ and $g\in V_{x_1}(A)$ such
that $|f|<1$ on $X\backslash U$ and $|g|<1$ on $X\backslash V$.
This easily implies that $\|\varphi(f,g)\|_X < \varphi(1,1)$ and
$\|\varphi(g,f)\|_X<\varphi(1,1)$, that is $\rho(Tf,Tg)=\rho(f,g)<
\rho(1,1)$. On the other hand, by Lemma \ref{iff}, we have $y_0\in
J^1_{x_0}$ and $y_0\in J^1_{x_1}$ which yield
$|Tf(y_0)|=1=|Tg(y_0)|$. Thus
\[ \rho(1,1)=\rho(Tf(y_0),Tg(y_0))\le \rho(Tf,Tg) =\rho(f,g)<\rho(1,1),\]
a contradiction. Consequently, there exists $x_0\in \Ch(\A)$ such
that for any $r>0$, $I_{y_0}^r \cap \Ch(\A)=I_{y_0}^1 \cap
\Ch(\A)=\{x_0\}$.

In a similar manner, we can show that $J_{x_0}^r \cap \Ch(\B)=\{y_0\}$ for all $ r>0 $.
\end{proof}

Using the above lemmas, we can define a function $\Phi: \Ch(\B)
\longrightarrow \Ch(\A)$ which associates to each $y_0\in \Ch(\B)$,
the unique point $x_0\in I^r_{y_0}\cap\Ch(\A)$ for all $r>0$.

\begin{lem}\label{romaxequ}
If $\rho= \rho_{\max}$, then $|Tf(y_0)|=|f(\Phi(y_0))|$ for all $f\in A$ and $y_0\in
\Ch(\B)$.
\end{lem}
\begin{proof}
The assertion is trivial for $f=0$ since $T$ is norm preserving.
Assume that $f\in A$ is nonzero and  $y_0\in \Ch(\B)$ such that
$|f(\Phi(y_0))|<|Tf(y_0)|$. Then for sufficiently small
$\epsilon>0$, we have $|f(\Phi(y_0))|+\epsilon<|Tf(y_0)|$. Using
Lemma \ref{romax}, there exists $h\in \|f\|_X V_{\Phi(y_0)}(A)$
satisfying
 \[ \rho_{\rm max}(f,h) \le \rho_{\rm max}(|f(\Phi(y_0))|+\epsilon, \|f\|_X).\]
Hence
\[\rho_{\rm max}(Tf(y_0),\|f\|_X)\le \rho_{\rm max}(Tf,Th)=\rho_{\rm max}(f,h) <\rho_{\rm max}(Tf(y_0), \|f\|_X),\]
which is impossible. Thus $|f(\Phi(y_0))|\ge|Tf(y_0)|$. The other
inequality is similarly proven. Consequently,
$|Tf(y_0)|=|f(\Phi(y_0))|$.
\end{proof}

\begin{lem}\label{ro+equ}
If $\rho=\rho_+$ and $\varphi$ satisfies either

{\rm (a)} $\varphi(t,0)=0=\varphi(0,t)$ for all $t\ge >0$, or

{\rm (b)} $\varphi(t,a)\to \infty$ and $\varphi(a,t)\to \infty$ as
$t\to \infty$ for all $a>0$, \\
then $|Tf(y_0)|=|f(\Phi(y_0))|$ for all $f\in A$ and
$y_0\in \Ch(\B)$.
\end{lem}
\begin{proof}
The assertion is again trivial for $f=0$, so we assume that $f\in
A$ is nonzero. Let $y_0\in \Ch(\B)$ and
$|f(\Phi(y_0))|<|Tf(y_0)|$. Then there exists $\epsilon >0$ such
that $ |f(\Phi(y_0))|+\epsilon <|Tf(y_0)|$.

Assume that  (a) holds. Then, using Lemma \ref{ro+}(i), we can find $h\in
V_{\Phi(y_0)}(A)$ such that
$\rho_+(f,h)<\rho_+(|f(\Phi(y_0))|+\epsilon, 1)$. Hence
\[\rho_+(Tf,Th)=\rho_+(f,h)< \rho_+(|f(\Phi(y_0))|+\epsilon,
1)<\rho_+(Tf(y_0),1),\] which is impossible since, by Lemma
\ref{iff}, $y_0\in J^1_{\Phi(y_0)}$, that is $|Th(y_0)|=1$.

Now assume that (b) holds. Then, using Lemma \ref{ro+}(ii), there
exist $\lambda>0$ and $h\in V_{\Phi(y_0)}(A)$ such that
$\rho_+(f,\lambda \, h)< \rho_+(|f(\Phi(y_0))|+\epsilon,
\lambda)$. We note that $\lambda h \in \lambda V_{\Phi(y_0)}(A)$,
and since, by Lemma \ref{iff}(i), $y_0\in J_{\Phi(y_0)}^\lambda$, we have
$|T(\lambda \,h)(y_0)|=\lambda$. Hence
\begin{align*}
\rho_+(Tf(y_0),\lambda)&=\rho_+(Tf(y_0),T(\lambda \,h)(y_0))\le\rho_+(Tf,T(\lambda \,h))\\
&=\rho_+(f,\lambda\, h)< \rho_+(|f(\Phi(y_0))|+\epsilon,\lambda)<\rho_+(Tf(y_0),\lambda),
\end{align*}
which is impossible.

We showed that in both cases (a) and (b), $|f(\Phi(y_0))|\ge
|Tf(y_0)|$. In the same manner, the other inequality is proven.
\end{proof}

{\bf Proof of Theorem \ref{main1}.} By the above lemmas, we need
only to show that the function $\Phi: \Ch(\B)\longrightarrow
\Ch(\A)$ is a homeomorphism. We first note that $\Phi$ is
surjective. Indeed, for each $x_0\in \Ch(\A)$, it follows from
Lemma \ref{singleton} that  there exists a point $y_0\in \Ch(\B)$
such that $J^1_{x_0}\cap \Ch(\B)= \{y_0\}$. Now, Lemma \ref{iff}
implies that $I^1_{y_0} \cap \Ch(\A)= \{x_0\}$. Since, by the
definition of $\Phi$, $I^1_{y_0}\cap \Ch(\A)=\{\Phi(y_0)\}$, we
have $\Phi(y_0)=x_0$,  i.e. $\Phi$ is surjective.

Similar argument shows that $\Phi$ is injective.

To prove that $\Phi$ is continuous, let $y_0\in \Ch(\B)$ and let
$U$ be an open neighborhood of $\Phi(y_0)$ in $X$. Choose  $h\in
V_{\Phi(y_0)}(A)$ with $|h|< \frac{1}{2}$ on $X \backslash U$ and
consider the open subset $V=\{ y\in \Ch(\B): |Th(y)|>
\frac{1}{2}\}$ of $\Ch(\B)$. Then, since $|Th|=|h\circ \Phi|$ on
$\Ch(\B)$,  we have $\Phi(V)\subseteq U\cap \Ch(\A)$. Hence $\Phi$
is continuous. Similarly, $\Phi^{-1}$ is also continuous. $\Box$

\section{Multiplicative semigroups of continuous functions}
In this section, we assume that $\varphi: \BC \times \BC
\longrightarrow \BR^+$ is  a continuous map satisfying (inc). We
consider $\rho_+$ and $\rho_{{\rm max}}$ as in Definition
\ref{dfn} and study surjections between certain multiplicative
semigroups of continuous functions which  preserve either $\rho_+$
or $\rho_{{\rm max}}$.

The main result of this section is as follows.
\begin{theorem}\label{main2}
Let $\varphi: \BC \times \BC \longrightarrow \BR^+$ be a
continuous map satisfying {\rm (inc)}. Let $X$ and $Y$ be locally
compact Hausdorff spaces, and let $A$ and $B$ be either
multiplicative subsets of $C_0(X)$ and $C_0(Y)$, respectively, or
positive parts of such subsets, which are also closed under
multiplication by strictly positive scalars. Suppose that
$\delta(A)=\Ch(\A)$ and $\delta(B)=\Ch(\B)$ for some function
algebras $ \A $ and $ \B$ on $X$ and $Y$, respectively, with $|A|
\subseteq |\A|$ and $|B|\subseteq |\B|$. Let $\rho\in \{\rho_+,
\rho_{\rm max}\}$ and $T: A\longrightarrow B$ be a surjective map
satisfying
\[ \rho(Tf,Tg)= \rho(f,g) \qquad (f,g\in A).  \]
Then $T$  induces a bijection $\Phi: \Ch(\B)\longrightarrow
\Ch(\A)$ between the Choquet boundaries of $\A$ and $\B$.
Moreover,

{\rm (i)} If $\rho=\rho_{\rm max}$, then $ \Phi $ is a homeomorphism
and $|Tf(y)|=|f(\Phi(y))|$ for all $f\in A$ and $y\in \Ch(\B)$.

{\rm (ii)} If $\rho= \rho_+$ and either
\begin{itemize}
    \item[(a)] $\varphi(t,0)=0=\varphi(0,t)$ for all $t\ge 0$, or

    \item[(b)] $\varphi(t,a)\to \infty$ and $\varphi(a,t)\to \infty$ as
    $t\to \infty$ for all $a>0$,
\end{itemize}
holds, then $ \Phi $ is a homeomorphism and
$|Tf(y)|=|f(\Phi(y))|$ for all $f\in A$ and $y\in \Ch(\B)$.
\end{theorem}
In what follows, we assume that $X,Y$ and $A,B$ are as in Theorem
\ref{main2}  and $T:A\longrightarrow B$ is a surjection satisfying
\[ \rho(Tf,Tg)= \rho(f,g) \qquad (f,g\in A),  \]
where $\rho \in \{\rho_+,\rho_{\rm max}\}$.

\begin{lem}\label{cap2}
{\rm (i)} For each $f\in A$, we have $\|Tf\|_Y=\|f\|_X$.

{\rm (ii)} Let $r>0$. Then for any multiplicative subset $D$ of
the sphere $S(A)=\{f\in A: \|f\|_X=1\}$,  we have $\cap_{f\in rD}
M(Tf)\neq \emptyset$. Similarly, for any multiplicative subset $D'$
of the sphere $S(B)=\{f\in B: \|g\|_Y=1\}$,  we have $\cap_{Tf\in
rD'} M(f)\neq \emptyset$.
\end{lem}
\begin{proof}
(i) The proof is similar to Lemma \ref{cap}(i).

 (ii) Let $D$ be a multiplicative subset of $S(A)$ and let $f_1,..., f_n\in rD$. Since
$D$ is multiplicative,
 we have $h= \Pi_{i=1}^n \frac{f_i}{r}\in D$.
  By (i), $\|Th\|_Y=\|h\|_X=1$.  Let  $y_0\in \Ch(\B)$ such
that  $|Th(y_0)|=1=\|h\|_X$.  We claim that $y_0\in M(Tf_i)$ for
$i=1,...,n$. Assume on the contrary that $|Tf_j(y_0)|<r$ for some
$1\le j\le n$. Then there exists a neighborhood $V$ of $y_0$ such
that $|Tf_j|<r$ on $V$ and there exists $h'\in A$ such that
$Th'(y_0)=1=\|Th'\|_Y$ and $|Th'|< 1$ on $Y \backslash V$. We note
that  for all $y\in V$, we have $|Tf_j(y)|<r $ and $|Th'(y)|\le 1$,
and for all $ y\in Y \backslash V$, we have $|Tf_j(y)|\le r$ and
$|Th'(y)|<1$. Hence, by  Lemma \ref{lem1}(iii), we get
$\varphi(Tf_j(y), Th'(y))< \varphi(r,1)$  and
$\varphi(Th'(y),Tf_j(y))< \varphi(1,r)$  for all $y\in Y$. Thus,
using Lemma \ref{lem1}(iv), we have $\|\varphi(Tf_j, Th')\|_Y<
\varphi(r,1)$ and $\|\varphi(Th',Tf_j)\|_Y< \varphi(1,r)$, that is
$\rho(Tf_j,Th')< \rho(r,1)$. The hypotheses imply that
\[  \rho(f_j,h') < \rho( r,1),\]
and consequently  for each $ x\in X $, at least one of the
inequalities $|f_j(x)|<r $ and $|h'(x)|<1$ holds. Now it follows
that for each $x\in X$,
\begin{eqnarray*}
\varphi(h(x),h'(x)) =\varphi(\frac{1}{r^n} f_j(x) \Pi_{i\neq j}
f_i(x), h'(x))<\varphi(1,1)
\end{eqnarray*}
and similarly $ \varphi(h'(x),h(x))<\varphi(1,1) $. Hence
\begin{align*}
    \rho(Th,Th')=\rho(h,h')<\rho(1,1),
\end{align*}
and consequently
\[ \rho(1,1)=\rho(Th(y_0),Th'(y_0))\le \rho(Th',Th)<\rho(1,1) \]
which is a contradiction. This argument shows that $y_0\in \cap_{i=1}^n M(Tf_i)$, as desired.

The other part has a similar proof.
\end{proof}

In this section, for $y_0\in \Ch(\B)$ and $r>0$, we set $I^r_{y_0}= \cap_{Tf\in
rF_{y_0}(B)}M(f)$. Similarly, for $x_0\in \Ch(\A)$ and $r>0$, we
set $J_{x_0}^r=\cap_{f\in rF_{x_0}(A)} M(Tf)$.

\begin{lem}\label{int2}
Let $y_0\in \Ch(\B)$ and $x_0\in \Ch(\A)$. Then  for all $r>0$, we
have $I^r_{y_0} \cap \Ch(\A)\neq \emptyset$ and $J^r_{x_0}\cap
\Ch(\B) \neq \emptyset$.
\end{lem}
\begin{proof}
 Let $y_0\in \Ch(\B)$ and let $r>0$. Since the set $F_{y_0}(B)$
is a multiplicative  subset of unit sphere $S(B)=\{g\in B:
\|g\|_Y=1\}$ of $B$, it follows from Lemma \ref{cap2}(ii) that
$I^r_{y_0}\neq \emptyset$. Choosing $z_0\in I^r_{y_0} $, we have
$z_0\in M(f)$ for all  $f\in A$ with $ Tf\in rF_{y_0}(B)$. By
assumption, for each $f\in A$ there exists $g\in \A$ with
$|f|=|g|$, which yields  $M(f)=M(g)$. Since any nonempty
intersection of peak sets of $\A$ intersects $\Ch(A)$,  as in
Lemma \ref{int}, we conclude that $I^r_{y_0} \cap
\Ch(\A)\neq\emptyset$.

Similarly, for each $x_0\in \Ch(\A)$, we have $J^r_{x_0}\cap \Ch(\B)
\neq \emptyset$.
\end{proof}

\begin{lem}
  For each $y_0\in \Ch(\B)$, there exists
  $x_0\in \Ch(\A)$ such that $I_{y_0}^1 \cap \Ch(\A)=\{x_0\}$ and $T(F_{x_0}(A))=F_{y_0}(B)$.
\end{lem}
\begin{proof}
    Let $y_0\in \Ch(\B)$ and let $x_0$ be an arbitrary point in $I_{y_0}^1\cap \Ch(\A)$.
    Then clearly,  $T^{-1}(F_{y_0}(B))\subseteq F_{x_0}(A)$. Since, by Lemma \ref{int2},  $J_{x_0}^1\cap \Ch(\B) \neq  \emptyset$,
    there exists a point $z_0\in \Ch(\B)$
    such that $T(F_{x_0}(A))\subseteq F_{z_0}(B)$. Thus
     \[ F_{y_0}(B)=T(T^{-1}(F_{y_0}(B)))\subseteq T(F_{x_0}(A))\subseteq
      F_{z_0}(B), \]  and hence $ y_0=z_0 $ since $y_0,z_0\in \delta(B)$. Therefore,
      $T(F_{x_0}(A))=F_{y_0}(B)$. In particular, $F_{x_0}(A)\subseteq
      T^{-1}(F_{y_0}(B))$. As it was noted before, the reverse inclusion also holds, and
      consequently we get $F_{x_0}(A) =T^{-1}(F_{y_0}(B))$.
      Since this equality holds for all $ x_0\in I_{y_0}^1\cap \Ch(\A)$,
      it follows that the intersection $I_{y_0}^1 \cap \Ch(\A)$ is the
      singleton $\{x_0\}$. Note that for this unique point $ x_0 $, we have $T(F_{x_0}(A))=F_{y_0}(B)$.
\end{proof}

Using the above lemma, we can define a bijective
    map $\Phi: \Ch(\B)\longrightarrow \Ch(\A)$ which associates to
    each $y_0\in \Ch(\B)$, the unique point $x_0\in I_{y_0}^1\cap\Ch(\A)$.

A minor modification of the proofs of Lemmas \ref{iff} and
\ref{singleton} yields the next lemma.

\begin{lem} \label{singleton2}
Let $y_0\in \Ch(B)$. Then for any $r>0$, we have
$I^r_{y_0}\cap\Ch(\A)=\{\Phi(y_0)\}$ and $J^r_{\Phi(y_0)}\cap\Ch(\B)=\{y_0\}$.
\end{lem}
The next two lemmas also have similar proofs to Lemmas
\ref{romaxequ} and \ref{ro+equ}, and hence we ignore their proofs.
\begin{lem} If $\rho= \rho_{\max}$, then $|Tf(y_0)|=|f(\Phi(y_0))|$ for all $f\in A$ and $y_0\in
\Ch(\B)$.
\end{lem}

\begin{lem}
If $\rho=\rho_+$ and $\varphi$ satisfies either

{\rm (a)} $\varphi(t,0)=0=\varphi(0,t)$ for all $t \ge 0$, or

{\rm (b)} $\varphi(t,a)\to \infty$ and $\varphi(a,t)\to \infty$ as
$t\to \infty$ for all $a>0$, \\
then $|Tf(y_0)|=|f(\Phi(y_0))|$ for all $f\in A$ and
$y_0\in \Ch(\B)$.
\end{lem}
The same proof as in Theorem \ref{main1} can be applied to show
that $\Phi$ is a homeomorphism. This completes the proof of
Theorem \ref{main2}.

\end{document}